\documentclass[11pt]{article}
\usepackage[english]{babel}
\usepackage{amssymb,amsmath,amsthm}
\newtheorem{theorem}{Theorem}[section]
\newtheorem{lemma}[theorem]{Lemma}
\newtheorem{proposition}[theorem]{Proposition}
\newtheorem{corollary}[theorem]{Corollary}
\newtheorem{question}[theorem]{Question}

{\theoremstyle{definition}}
{\theoremstyle{definition}\newtheorem{example}[theorem]{Example}}
{\theoremstyle{definition}\newtheorem{definition}[theorem]{Definition}}
{\theoremstyle{definition}\newtheorem{remark}[theorem]{Remark}}

\newtheorem*{thmdu}{Theorem U}
\newtheorem*{thmd}{Proposition G}

\numberwithin{equation}{section}

\font\goth=eufm10 scaled 1200
\def\C{{\mathbb C}}
\def\N{{\mathbb N}}
\def\Z{{\mathbb Z}}
\def\R{{\mathbb R}}
\def\T{{\mathbb T}}
\def\D{{\mathbb D}}
\def\M{\hbox{\goth M}}
\def\K{{\mathbb K}}
\def\F{{\mathcal F}}
\def\H{{\mathcal H}}
\def\epsilon{\varepsilon}
\def\kappa{\varkappa}
\def\phi{\varphi}
\def\leq{\leqslant}
\def\geq{\geqslant}
\def\dim{\hbox{\tt dim}\,}
\def\ker{\hbox{\tt ker}\,}
\def\Ker{\hbox{\tt ker}^\star\,}
\def\KER{\hbox{\tt ker}^\dagger\,}
\def\spann{\hbox{\tt span}\,}
\def\ssub#1#2{#1_{{}_{{\scriptstyle #2}}}}
\def\llll{\vrule width0pt height16pt depth12 pt}
\def\dimr{\ssub{\dim}{\R}}

\title{Hypercyclic and mixing operator semigroups}

\author{Stanislav Shkarin}

\date{}

\begin{document}

\maketitle

\begin{abstract}  We describe a class of topological vector spaces
admitting a mixing uniformly continuous operator group
$\{T_t\}_{t\in\C^n}$ with holomorphic dependence on the parameter
$t$. This result covers those existing in the literature. We also
describe a class of topological vector spaces admitting no
supercyclic strongly continuous operator semigroups $\{T_t\}_{t\geq
0}$.
\end{abstract}

\small \noindent{\bf MSC:} \ \ 47A16, 37A25

\noindent{\bf Keywords:} \ \ Hypercyclic operators; supercyclic
operators; hypercyclic semigroups; mixing semigroups \normalsize

\section{Introduction \label{s1}}\rm

Unless stated otherwise, all vector spaces in this article are over
the field $\K$, being either the field $\C$ of complex numbers or
the field $\R$ of real numbers and all topological spaces {\it are
assumed to be Hausdorff}. As usual, $\Z$ is the set of integers,
$\Z_+$ is the set of non-negative integers, $\N$ is the set of
positive integers and $\R_+$ is the set of non-negative real
numbers. Symbol $L(X,Y)$ stands for the space of continuous linear
operators from a topological vector space $X$ to a topological
vector space $Y$. We write $L(X)$ instead of $L(X,X)$ and $X'$
instead of $L(X,\K)$. $X'_\sigma$ is $X'$ with the weak topology
$\sigma$, being the weakest topology on $X'$ making the maps
$f\mapsto f(x)$ from $X'$ to $\K$ continuous for all $x\in X$. For
any $T\in L(X)$, the dual operator $T':X'\to X'$ is defined as
usual: $(T'f)(x)=f(Tx)$ for $f\in X'$ and $x\in X$. Clearly $T'\in
L(X'_\sigma)$. For a subset $A$ of a vector space $X$, $\spann(A)$
stands for the linear span of $A$. For brevity, we say {\it locally
convex space} for a locally convex topological vector space. A
subset $B$ of a topological vector space $X$ is called {\it bounded}
if for any neighborhood $U$ of zero in $X$, a scalar multiple of $U$
contains $B$. The topology $\tau$ of a topological vector space $X$
is called {\it weak} if $\tau$ is exactly the weakest topology
making each $f\in Y$ continuous for some linear space $Y$ of linear
functionals on $X$ separating points of $X$. An $\F$-space is a
complete metrizable topological vector space. A locally convex
$\F$-space is called a {\it Fr\'echet} space. Symbol $\omega$ stands
for the space of all sequences $\{x_n\}_{n\in\Z_+}$ in $\K$ with
coordinatewise convergence topology. We denote the linear subspace
of $\omega$ consisting of sequences $x$ with finite support
$\{n\in\Z_+:x_n\neq 0\}$ by $\phi$. If $X$ is a topological vector
space, then $A\subset X'$ is called {\it equicontinuous} if there is
a neighborhood $U$ of zero in $X$ such that $|f(x)|\leq 1$ for any
$x\in U$ and $f\in A$.

Let $X$ and $Y$ be topological spaces and $\{T_a:a\in A\}$ be a
family of continuous maps from $X$ to $Y$. An element $x\in X$ is
called {\it universal} for this family if $\{T_ax:a\in A\}$ is dense
in $Y$ and $\{T_a:a\in A\}$ is said to be {\it universal} if it has
a universal element. An {\it operator semigroup} on a topological
vector space $X$ is a family $\{T_t\}_{t\in A}$ of operators from
$L(X)$ labeled by elements of an abelian monoid $A$ and satisfying
$T_0=I$, $T_{s+t}=T_tT_s$ for any $t,s\in A$. A {\it norm} on $A$ is
a function $|\cdot|:A\to[0,\infty)$ satisfying $|na|=n|a|$ and
$|a+b|\leq |a|+|b|$ for any $n\in\Z_+$ and $a,b\in A$. An abelian
monoid equipped with a norm is a {\it normed semigroup}. We are
mainly concerned with the case when $A$ is a closed additive
subsemigroup of $\R^k$ containing $0$ with the norm $|a|$ being the
Euclidean distance from $a$ to $0$. In the latter case $A$ carries
the topology inherited from $\R^k$ and an operator semigroup
$\{T_t\}_{t\in A}$ is called {\it strongly continuous} if the map
$t\mapsto T_tx$ from $A$ to $X$ is continuous for any $x\in X$. We
say that an operator semigroup $\{T_t\}_{t\in A}$ is {\it uniformly
continuous} if there is a neighborhood $U$ of zero in $X$ such that
for any sequence $\{t_n\}_{n\in\Z_+}$ in $A$ converging to $t\in A$,
$T_{t_n}x$ converges to $T_tx$ uniformly on $U$. Clearly, uniform
continuity is strictly stronger than strong continuity. If $A$ is a
normed semigroup and $\{T_t\}_{t\in A}$ is an operator semigroup on
a topological vector space $X$, then we say that $\{T_t\}_{t\in A}$
is {\it mixing} if for any non-empty open subsets $U,V$ of $X$,
there is $r>0$ such that $T_t(U)\cap V\neq\varnothing$ provided
$|t|>r$. We say that $\{T_t\}_{t\in A}$ is {\it hypercyclic}
(respectively, {\it supercyclic}) if the family $\{T_t:t\in A\}$
(respectively, $\{zT_t:z\in\K,\ t\in A\}$) is universal.
$\{T_t\}_{t\in A}$ is said to be {\it hereditarily hypercyclic} if
for any sequence $\{t_n\}_{n\in\Z_+}$ in $A$ satisfying
$|t_n|\to\infty$, $\{T_{t_n}:n\in\Z_+\}$ is universal. $T\in L(X)$
is called {\it hypercyclic, supercyclic, hereditarily hypercyclic
{\rm or} mixing} if the semigroup $\{T^n\}_{n\in\Z_+}$ has the same
property. Hypercyclic and supercyclic operators have been intensely
studied during last few decades, see \cite{bama-book} and references
therein. Recall that a topological space $X$ is called a {\it Baire
space} if the intersection of countably many dense open subsets of
$X$ is dense in $X$. By the classical Baire theorem, complete metric
spaces are Baire.

\begin{proposition} \label{untr2} Let $X$ be a
topological vector space and $A$ be a normed semigroup. Then any
hereditarily hypercyclic operator semigroup $\{T_a\}_{a\in A}$ on
$X$ is mixing. If $X$ is Baire separable and metrizable, then the
converse implication holds$:$ any mixing operator semigroup
$\{T_a\}_{a\in A}$ on $X$ is hereditarily hypercyclic.
\end{proposition}

The above proposition is a combination of well-known facts,
appearing in the literature in various modifications. In the next
section we prove it for sake of completeness. It is worth noting
that for any subsemigroup $A_0$ of $A$, not lying in the kernel of
the norm, $\{T_t\}_{t\in A_0}$ is mixing if $\{T_t\}_{t\in A}$ is
mixing. In particular, if $\{T_t\}_{t\in A}$ is mixing, then $T_t$
is mixing whenever $|t|>0$.

The question of existence of supercyclic or hypercyclic operators or
semigroups on various types of topological vector spaces was
intensely studied. The fact that there are no hypercyclic operators
on any finite dimensional topological vector space goes back to
Rolewicz \cite{rol}. The last result in this direction is due to
Wengenroth \cite{ww}, who proved that a hypercyclic operator on any
topological vector space (locally convex or not) has no closed
invariant subspaces of positive finite codimension, while any
supercyclic operator has no closed invariant subspaces of finite
$\R$-codimension $>2$. In particular, his result implies the
(already well known by then) fact that there are no supercyclic
operators on a finite dimensional topological vector space of
$\R$-dimension $>2$. Herzog \cite{ger} proved that there is a
supercyclic operator on any separable infinite dimensional Banach
space. Ansari \cite{ansa1} and Bernal-Gonz\'alez \cite{bernal},
answering a question raised by Herrero, showed independently that
any separable infinite dimensional Banach space supports a
hypercyclic operator. Using the same idea as in \cite{ansa1}, Bonet
and Peris \cite{bonper} proved that there is a hypercyclic operator
on any separable infinite dimensional Fr\'echet space and
demonstrated that there is a hypercyclic operator on the inductive
limit $X$ of a sequence $\{X_n\}_{n\in\Z_+}$ of separable Banach
spaces provided $X_0$ is dense in $X$. Grivaux \cite{gri} observed
that hypercyclic operators $T$ in \cite{ansa1,bernal,bonper} are
mixing and therefore hereditarily hypercyclic. They actually come
from the same source. Namely, according to Salas \cite{sal}, an
operator of the shape $I+T$, where $T$ is a backward weighted shift
on $\ell_1$, is hypercyclic. Virtually the same proof demonstrates
that these operators are mixing.  Moreover, all operators
constructed in the above cited papers are hypercyclic or mixing
because of a quasisimilarity with an operator of the shape identity
plus a backward weighted shift. A similar idea was used by
Berm\'udez, Bonilla and Martin\'on \cite{ex1} and Bernal-Gonz\'alez
and Grosse-Erdmann \cite{ex2}, who proved that any separable
infinite dimensional Banach space supports a hypercyclic strongly
continuous semigroup $\{T_t\}_{t\in \R_+}$. Berm\'udez, Bonilla,
Conejero and Peris \cite{holo} proved that on any separable infinite
dimensional complex Banach space $X$, there is a mixing strongly
continuous semigroup $\{T_t\}_{t\in \C}$ such that the map $t\mapsto
T_t$ is holomorphic. Finally, Conejero \cite{cone} proved that any
separable infinite dimensional complex Fr\'echet space $X$
non-isomorphic to $\omega$ supports a mixing operator semigroup
$\{T_t\}_{t\in\R_+}$ such that $T_{t_n}x$ uniformly converges to
$T_tx$ for $x$ from any bounded subset of $X$ whenever $t_n\to t$.

\begin{definition}\label{MMM} We say that a topological vector space
$X$ belongs to the class $\M_0$ if there is a dense subspace $Y$ of
$X$ admitting a topology $\tau$ stronger than the one inherited from
$X$ and such that $(Y,\tau)$ is a separable $\F$-space. We say that
$X$ belongs to $\M_1$ if there is a linearly independent
equicontinuous sequence $\{f_n\}_{n\in\Z_+}$ in $X'$. Finally,
$\M=\M_0\cap \M_1$.
\end{definition}

\begin{remark}\label{Rem3} Obviously, $X\in \M_1$  if and
only if there exists a continuous seminorm $p$ on $X$ such that
$\ker p=p^{-1}(0)$ has infinite codimension in $X$. In particular, a
locally convex space $X$ belongs to $\M_1$ if and only if its
topology is not weak.
\end{remark}

\subsection{Results}

The following theorem extracts the maximum of the method both in
terms of the class of spaces and semigroups. Although the general
idea remains the same, the proof requires dealing with a number of
technical details of various nature.

\begin{theorem}\label{saan} Let $X\in\M$. Then for any $k\in\N$,
there exists a uniformly continuous hereditarily hypercyclic
$(\!$and therefore mixing$)$ operator group $\{T_t\}_{t\in\K^k}$ on
$X$ such that the map $z\mapsto f(T_zx)$ from $\K^k$ to $\K$ is
analytic for each $x\in X$ and $f\in X'$.
\end{theorem}

Since for any hereditarily hypercyclic semigroup
$\{T_t\}_{t\in\K^k}$ and any non-zero $t\in\K^k$, $T_t$ is
hereditarily hypercyclic, Theorem~\ref{saan} provides a hereditarily
hypercyclic operator on each $X\in\M$. Obviously, any separable
$\F$-space belongs to $\M_0$. It is well-known \cite{shifer} that
the topology on a Fr\'echet space $X$ differs from the weak topology
if and only if $X$ is infinite dimensional and it is non-isomorphic
to $\omega$. Thus any separable infinite dimensional Fr\'echet space
non-isomorphic to $\omega$ belongs to $\M$. The latter fact is also
implicitly contained in \cite{bonper}. Similarly, an infinite
dimensional inductive limit $X$ of a sequence $\{X_n\}_{n\in\Z_+}$
of separable Banach spaces belongs to $\M$ provided $X_0$ is dense
in $X$. Thus all the above mentioned existence theorems are
particular cases of Theorem~\ref{saan}. The following proposition
characterizes $\F$-spaces in the class $\M$.

\begin{proposition}\label{fsp} Let $X$ be an $\F$-space. Then $X$
belongs to $\M$ if and only if $X$ is separable and the algebraic
dimension of $X'$ is uncountable.
\end{proposition}

Proposition~\ref{fsp} ensures that Theorem~\ref{saan} can be applied
to a variety of not locally convex $\F$-spaces including $\ell_p$
with $0<p<1$. We briefly outline the main idea of the proof of
Theorem~\ref{saan} because it is barely recognizable in the main
text, where the intermediate results are presented in much greater
generality than strictly necessary. Consider the completion of the
$k^{\rm th}$ projective tensor power of $\ell_1$:
$X=\ell_1\widehat{\otimes}{\dots}\widehat{\otimes}\ell_1$ and
$T_1,\dots,T_k\in L(X)$ of the shape $T_j=I\otimes{\dots}\otimes
I\otimes S_j\otimes I\otimes{\dots}\otimes I$, where $S_j\in
L(\ell_1)$ is a backward weighted shift sitting in $j^{\rm th}$
place. Since $T_j$ are pairwise commuting, we have got a uniformly
continuous operator group $\{e^{\langle z,T\rangle}\}_{z\in\K^k}$ on
$X$, where $\langle z,T\rangle=z_1T_1+{\dots}+z_kT_k$. We show that
$\{e^{\langle z,T\rangle}\}_{z\in\K^k}$ is hereditarily hypercyclic.
The class $\M$ turns out to be exactly the class of topological
vector spaces to which such a group can be transferred by means of
quasisimilarity.

The following theorem is kind of an opposite of Theorem~\ref{saan}.

\begin{theorem}\label{omegG}
There are no supercyclic strongly continuous operator semigroups
$\{T_t\}_{t\in\R_+}$ on a topological vector space $X$ if either
$2<\dimr X<2^{\aleph_0}$ or $2<\dimr X'<2^{\aleph_0}$.
\end{theorem}

Since $\dim\omega'=\aleph_0$, Theorem~\ref{omegG} implies that there
are no supercyclic strongly continuous operator semigroups
$\{T_t\}_{t\in\R_+}$ on $\omega$, which is a stronger version of a
result in \cite{cone}. This observation together with
Theorem~\ref{saan} imply the following curious result.

\begin{corollary}\label{COR} For a separable infinite
dimensional Fr\'echet space $X$, the following are equivalent$:$
\begin{itemize}\itemsep=-2pt
\item[\rm(\ref{COR}.1)]for each $k\in\N$, there is a mixing uniformly
continuous operator group $\{T_t\}_{t\in \R^k}$ on $X;$
\item[\rm(\ref{COR}.3)]there is a supercyclic strongly continuous
operator semigroup $\{T_t\}_{t\in \R_+}$ on $X;$
\item[\rm(\ref{COR}.4)]$X$ is non-isomorphic to $\omega$.
\end{itemize}
\end{corollary}

\section{Extended backward shifts \label{ebsh}}

Godefroy and Shapiro \cite{gs} introduced the notion of a
generalized backward shift. Namely, a continuous linear operator $T$
on a topological vector space $X$ is called a {\it generalized
backward shift} if the union of $\ker T^n$ for $n\in\N$ is dense in
$X$ and $\ker T$ is one-dimensional. We say that $T$ is an {\it
extended backward shift} if the linear span of the union of
$T^n(\ker T^{2n})$ is dense in $X$. Using an easy dimension argument
\cite{gs} one can show that any generalized backward shift is an
extended backward shift. It is worth noting
\cite[Theorem~2.2]{bama-book} that for any extended backward shift
$T$, $I+T$ is mixing. We need a multi-operator analog of this
concept.

Let $X$ be a topological vector space. We say that
$T=(T_1,\dots,T_k)\in L(X)^k$ is a {\it EBS$_k$-tuple} if
$T_mT_j=T_jT_m$  for $1\leq j,m\leq k$ and $\KER (T)$ is dense in
$X$, where
\begin{equation}\label{KERk}
\llll\smash{\KER (T)=\spann\bigcup_{n\in\N^k} \kappa(n,T)\qquad
\text{and}\qquad \kappa(n,T)=T_1^{n_1}\!\dots
T_k^{n_k}\biggl(\bigcap_{j=1}^k \ker T_j^{2n_j}\!\biggr).}
\end{equation}

\subsection{Shifts on finite dimensional spaces \label{tbs}}

The following two lemmas are implicitly contained in the proof of
Theorem~5.2 in \cite{dsw}. For sake of convenience, we provide their
proofs.

\begin{lemma}\label{detan}
For each $n\in\N$ and $z\in\C\setminus\{0\}$, the matrix
$A_{n,z}=\left\{\frac{z^{j+k-1}}{(j+k-1)!}\right\}_{j,k=1}^n$ is
invertible.
\end{lemma}

\begin{proof} Invertibility of $A_{n,1}$ is proved in
\cite[Lemma~2.7]{bama-book}. For $z\in\C$, consider the diagonal
$n\times n$ matrix $D_{n,z}$ with the entries $(1,z,\dots,z^{n-1})$
on the main diagonal. Clearly
\begin{equation}\label{dia}
A_{n,z}=zD_{n,z}A_{n,1}D_{n,z} \ \ \ \text{for any $z\in\C$}.
\end{equation}
Since $A_{n,1}$ and $D_{n,z}$ for $z\neq 0$ are invertible,
$A_{n,z}$ is invertible for any $n\in\N$ and $z\in\C\setminus\{0\}$.
\end{proof}

\begin{lemma}\label{jordan}
Let $n\in\N$, $e_1,\dots,e_{2n}$ be the canonical basis of
$\K^{2n}$, $S\in L(\K^{2n})$ be defined by $Se_1=0$ and
$Se_k=e_{k-1}$ for $2\leq k\leq 2n$ and $P$ be the linear projection
on $\K^{2n}$ onto $E=\spann\{e_1,\dots,e_n\}$ along
$F=\spann\{e_{n+1},\dots,e_{2n}\}$. Then for any
$z\in\K\setminus\{0\}$ and $u,v\in E$, there exists a unique
$x^z=x^z(u,v)\in \K^{2n}$ such that
\begin{equation}\label{ab}
Px^z=u\quad\text{and}\quad Pe^{zS}x^z=v.
\end{equation}
Moreover, for any bounded subset $B$ of $E$ and any $\epsilon>0$,
there is $c=c(\epsilon,B)>0$ such that
\begin{align}
\sup_{u,v\in B}|(x^z(u,v))_{n+j}|&\leq c|z|^{-j}\quad\text{for
$1\leq j\leq n$ and $|z|\geq\epsilon;$} \label{Ox1}
\\
\sup_{u,v\in B}|(e^{zS}x^z(u,v))_{n+j}|&\leq c|z|^{-j}\quad\text{for
$1\leq j\leq n$ and $|z|\geq\epsilon$.} \label{Ox2}
\end{align}
In particular, $x^z(u,v)\to u$ and $e^{zS}x^z(u,v)\to v$ as $|z|\to
\infty$ uniformly for $u$ and $v$ from any bounded subset of $E$.
\end{lemma}

\begin{proof} Let $u,v\in E$ and $z\in\K\setminus\{0\}$.
For $y\in \K^{2n}$ we denote
$\overline{y}=(y_{n+1},\dots,y_{2n})\in\K^n$. One easily sees that
(\ref{ab}) is equivalent to the vector equation
\begin{equation}
A_{n,z}\overline{x}^z=w^z, \label{vector}
\end{equation}
where $A_{n,z}$ is the matrix from Lemma~\ref{detan} and
$w^z=w^z(u,v)\in\K^n$ is defined as
\begin{equation} \label{wm}
w^z_j=v_{n-j+1}-\sum_{k=n-j+1}^n\frac{z^{k+j-n-1}u_k}{(k+j-n-1)!}
\quad\text{for $1\leq j\leq n$},
\end{equation}
provided we set $x_{j}=u_{j}$ for $1\leq j\leq n$. By Lemma
\ref{detan}, $A_{n,z}$ is invertible for any $z\neq 0$ and therefore
(\ref{vector}) is uniquely solvable. Thus there exists a unique
$x^z=x^z(u,v)\in \K^{2n}$ satisfying (\ref{ab}). It remains to
verify (\ref{Ox1}) and (\ref{Ox2}). By (\ref{wm}), for any bounded
subset $B$ of $E$ and any $\epsilon>0$, there is $a=a(\epsilon,B)>0$
such that
\begin{equation}
|(w^z(u,v))_j|\leq a|z|^{j-1}\quad\text{if $u,v\in B$,
$|z|\geq\epsilon$ and $1\leq j\leq n$}. \label{Ow}
\end{equation}
By (\ref{Ow}), $\{D_{n,z}^{-1}w^z(u,v):|z|\geq \epsilon,\ u,v\in
B\}$ and therefore $Q=\{A_{n,1}^{-1}D_{n,z}^{-1}w^z(u,v):|z|\geq
\epsilon,\ u,v\in B\}$ are bounded in $\K^n$. Since by
(\ref{vector}) and (\ref{dia}), $\overline{x}^z=A_{n,z}^{-1}w^z=
z^{-1}D_{n,z}^{-1}A_{n,1}^{-1}D_{n,z}^{-1}w^z$, we have
$$
(x^z(u,v))_{n+j}=\overline{x}^z_j\subseteq \{z^{-1}(D_{n,z}^{-1}
y)_j:y\in Q\}\ \ \text{if $|z|\geq \epsilon$,\ \ and\ \ $u,v\in B$}.
$$
Boundedness of $Q$ implies that (\ref{Ox1}) is satisfied with some
$c=c_1(\epsilon,B)$. Finally, since for $1\leq j\leq n$, we have
$(e^{zS}x^z)_{n+j}=\sum\limits_{l=n+j}^{2n}
\frac{z^{l-n-j}x^z_l}{(l-n-j)!}$, there is $c=c_2(\epsilon,B)$ for
which (\ref{Ox2}) is satisfied. Hence (\ref{Ox2}) and (\ref{Ox1})
hold with $c=\max\{c_1,c_2\}$.
\end{proof}

\begin{corollary}\label{jordan1}
Let $n\in\N$, $E\subseteq \K^{2n}$ and $S\in L(\K^{2n})$ be as in
Lemma~$\ref{jordan}$. Then for any $u,v\in E$ and any sequence
$\{z_j\}_{j\in\Z_+}$ in $\K$ satisfying $|z_j|\to\infty$, there
exists a sequence $\{x_j\}_{j\in\Z_+}$ in $\K^{2n}$ such that
$x_j\to u$ and $e^{z_jS}x_j\to v$.
\end{corollary}

We need the following multi-operator version of
Corollary~\ref{jordan1}.

\begin{lemma}\label{jordan2} Let $k\in\N$, $n_1,\dots,n_k\in \N$,
for each $j\in\{1,\dots,k\}$ let $e^j_1,\dots,e^j_{2n_j}$ be the
canonical basis in $\K^{2n_j}$,
$E_j=\spann\{e^j_1,\dots,e^j_{n_j}\}$ and $S_j\in L(\K^{2n_j})$ be
the backward shift: $S_je^j_1=0$ and $S_je^j_l=e^j_{l-1}$ for $2\leq
l\leq 2n_j$. Let also $X=\K^{2n_1}\otimes {\dots}\otimes \K^{2n_k}$,
$E=E_1\otimes {\dots}\otimes E_k$ and
$$
T_j\in L(X),\quad T_j=I\otimes{\dots}I\otimes S_j\otimes I\otimes
{\dots}\otimes I\ \ \ \text{for $1\leq j\leq k$},
$$
where $S_j$ sits in the $j^{\rm th}$ place. Finally, let
$\{z_m\}_{m\in\Z_+}$ be a sequence in $\K^k$ satisfying
$|z_m|\to\infty$. Then for any $u,v\in E$, there exists a sequence
$\{x_m\}_{m\in\Z_+}$ in $X$ such that $x_m\to u$ and $e^{\langle
z_m,T\rangle}x_m\to v$, where $\langle
s,T\rangle=s_1T_1+{\dots}+s_kT_k$.
\end{lemma}

\begin{proof} Let $\overline{\K}=\K\cup\{\infty\}$ be the one-point
compactification of $\K$. Clearly it is enough to show that any
sequence $\{w_m\}$ in $\K^k$ satisfying $|w_m|\to\infty$ has a
subsequence $\{z_m\}$ for which the statement of the lemma is true.
Since $\overline{\K}^k$ is compact and metrizable, we can, without
loss of generality, assume that $\{z_m\}$ converges to $w\in
\overline{\K}^k$. Since $|z_m|\to\infty$, the set
$C=\{j:w_j=\infty\}$ is non-empty. Without loss of generality, we
may also assume that $C=\{1,\dots,r\}$ with $1\leq r\leq k$.

Denote by $\Sigma$ the set of $(u,v)\in X^2$ for which there is a
sequence $\{x_m\}_{m\in\Z_+}$ in $X$ such that $x_m\to u$ and
$e^{\langle z_m,T\rangle}x_m\to v$. In this notation, the statement
of the lemma is equivalent to the inclusion $E\times E\subseteq
\Sigma$. Let $u_j\in E_j$ for $1\leq j\leq k$ and
$u=u_1\otimes{\dots}\otimes u_k$. By Corollary~\ref{jordan1}, there
exist sequences $\{x_{j,m}\}_{m\in\Z_+}$ and
$\{y_{j,m}\}_{m\in\Z_+}$ in $\K^{2n_j}$ such that
$$
\text{$x_{j,m}\to 0$, $e^{(z_m)_j S_j}x_{j,m}\to u_j$, $y_{j,m}\to
u_j$ and $e^{(z_m)_j S_j}y_{j,m}\to 0$ for $1\leq j\leq r$}.
$$
We put $x_{j,m}=e^{-w_jS_j}u_j$ and $y_{j,m}=u_j$ for $r<j\leq k$
and $m\in\Z_+$. Consider the sequences $\{x_m\}_{m\in\Z_+}$ and
$\{y_m\}_{m\in\Z_+}$ in $X$ defined by
$x_m=x_{1,m}\otimes{\dots}\otimes x_{k,m}$ and
$y_m=y_{1,m}\otimes{\dots}\otimes y_{k,m}$. By definition of $x_m$
and $y_m$ and the above display, $x_m\to 0$ and $y_m\to u$. For
instance, $x_m\to 0$ because $\{x_{j,m}\}$ are bounded and
$x_{1,m}\to 0$. Similarly, taking into account that $(z_m)_j\to w_j$
for $j>r$, we see that $e^{\langle z_m,T\rangle}x_m\to u$ and
$e^{\langle z_m,T\rangle}y_m\to 0$. Hence $(u,0)\in \Sigma$ and
$(0,u)\in\Sigma$. Thus $(\{0\}\times E_0)\cup(E_0\times
\{0\})\subseteq \Sigma$, where $E_0=\{u_1\otimes{\dots}\otimes
u_k:u_j\in E_j,\ 1\leq j\leq k\}$. On the other hand,
$\spann(\{0\}\times E_0)\cup(E_0\times \{0\})=E\times E$. Since
$\Sigma$ is a linear space, $E\times E\subseteq \Sigma$.
\end{proof}

For applications it is more convenient to reformulate the above
lemma in the coordinate form.

\begin{corollary}\label{jordan3} Let $k\in\N$, $n_1,\dots,n_k\in \N$,
$N_j=\{1,\dots,2n_j\}$ and $Q_j=\{1,\dots,n_j\}$ for $1\leq j\leq
k$. Consider $M=N_1\times {\dots}\times N_k$ and $M_0=Q_1\times
{\dots}\times Q_k$, let $\{e_m:m\in M\}$ be the canonical basis of
$X=\K^M$ and $E=\spann\{e_m:m\in M_0\}$. For $1\leq j\leq k$, let
$T_j\in L(X)$ be defined by $T_je_m=0$ if $m_j=1$ and
$T_je_m=e_{m'}$ if $m_j>1$, where $m'_l=m_l$ if $l\neq j$,
$m'_j=m_j-1$. Then for any sequence $\{z_m\}_{m\in\Z_+}$ in $\K^k$
satisfying $|z_m|\to\infty$ and any $u,v\in E$, there is a sequence
$\{x_m\}_{m\in\Z_+}$ in $X$ such that $x_m\to u$ and $e^{\langle
z_m,T\rangle}x_m\to v$, where $\langle
s,T\rangle=s_1T_1+{\dots}+s_kT_k$.
\end{corollary}

\subsection{The key lemma \label{key}}

\begin{lemma}\label{kerim0}
Let $X$ be a topological vector space, $k\in\N$, $n\in\N^k$ and
$A\in L(X)^k$ be such that $A_jA_l=A_lA_j$ for $1\leq l,j\leq k$.
Then for each $x$ from $\kappa(n,A)$ defined in $(\ref{KERk})$,
there is a common finite dimensional invariant subspace $Y$ for
$A_1,\dots,A_k$ such that for any sequence $\{z_m\}_{m\in\Z_+}$ in
$\K^k$ satisfying $|z_m|\to\infty$, there exist sequences
$\{x_m\}_{m\in\Z_+}$ and $\{y_m\}_{m\in\Z_+}$ in $Y$ for which
\begin{equation}\label{converg0}
x_m\to0,\ \ e^{A_{z_m}}x_m\to x,\ \ y_m\to x\ \ \text{and}\ \
e^{A_{z_m}}y_m\to 0,\ \ \text{where
$A_s=(s_1A_1+{\dots}+s_kA_k)\bigr|_Y$}.
\end{equation}
\end{lemma}

\begin{proof}
Since $x\in \kappa(n,T)$, there is $y\in X$ such that
$x=A_1^{n_1}\!\dots A_k^{n_k}y$ and $A_j^{2n_j}y=0$ for $1\leq j\leq
k$. Let $N_j=\{1,\dots,2n_j\}$ and $Q_j=\{1,\dots,n_j\}$ for $1\leq
j\leq k$. Denote $M=N_1\times {\dots}\times N_{k}$ and
$M_0=Q_1\times {\dots}\times Q_{k}$. Let $h_l=A_1^{2n_1-l_1}\!\dots
A_k^{2n_{k}-l_{k}}y$ for $l\in M$ and $Y=\spann\{h_l:l\in M\}$.
Clearly $Y$ is finite dimensional and $A_j h_l=0$ if $l_j=1$, $A_j
h_l=h_{l'}$ if $l_j>1$, where $l'_r=l_r$ for $r\neq j$ and
$l'_j=l_j-1$. Hence $Y$ is invariant for each $A_j$. Consider $J\in
L(\K^M,Y)$ defined by $Je_l=h_l$ for $l\in M$. Let also
$E=\spann\{e_l:l\in M_0\}$ and $T_j\in L(\K^M)$ be as in
Corollary~\ref{jordan3}. Taking into account the definition of $T_j$
and the action of $A_j$ on $h_l$, we see that $A_jJ=JT_j$ for $1\leq
j\leq k$. Clearly $n\in M_0$ and therefore $e_n\in E$. Since
$x=A_1^{n_1}\!\dots A_k^{n_k}y$, we have $x=h_n$. By
Corollary~\ref{jordan3}, there exist sequences $\{u_m\}_{m\in\Z_+}$
and $\{v_m\}_{m\in\Z_+}$ in $\K^M$ such that $u_m\to e_n$,
$e^{\langle z_m,T\rangle}u_m\to 0$, $v_m\to 0$ and $e^{\langle
z_m,T\rangle}u_m\to e_n$. Now let $y_m=Ju_m$ and $x_m=Jv_m$ for
$m\in\Z_+$. Then $\{x_m\}$ and $\{y_m\}$ are sequences in $Y$. From
the relations $A_jJ=JT_j$ and the fact that $\K^M$ and $Y$ are
finite dimensional, it follows that $x_m\to 0$, $y_m\to Je_n=x$,
$e^{A_{z_m}}x_m\to Je_n=x$ and $e^{A_{z_m}}y_m\to 0$. Thus
(\ref{converg0}) is satisfied.
\end{proof}

From now on, if $A=(A_1,\dots,A_k)$ is a $k$-tuple of continuous
linear operators on a topological vector space $X$ and $z\in\K^k$,
we write
$$
\langle z,A\rangle=z_1A_1+{\dots}+z_kA_k.
$$
We also use the following convention. Let $X$ be a topological
vector space and $S\in L(X)$. By saying that {\it $e^S$ is
well-defined}, we mean that for each $x\in X$, the series
$\sum\limits_{n=0}^\infty \frac1{n!}S^nx$ converges in $X$ and
defines a continuous linear operator denoted $e^S$.

\begin{corollary}\label{kerim1}
Let $X$ be a topological vector space, $k\in\N$ and $A\in L(X)^k$ be
a $k$-tuple of pairwise commuting operators such that for any $z\in
\K^k$, $e^{\langle z,A\rangle}$ is well-defined. Then for each $x$
and $y$ from the space $\KER(A)$ defined in $(\ref{KERk})$ and any
sequence $\{z_m\}_{m\in\Z_+}$ in $\K^k$ satisfying $|z_m|\to\infty$,
there is a sequence $\{u_m\}_{m\in\Z_+}$ in $X$ such that $u_m\to x$
and $e^{\langle z_m,A\rangle}u_m\to y$.
\end{corollary}

\begin{proof} Fix a sequence $\{z_m\}_{m\in\Z_+}$ in
$\K^k$ satisfying $|z_m|\to\infty$. Let $\Sigma$ be the set of
$(x,y)\in X^2$ for which there exists a sequence
$\{u_m\}_{n\in\Z_+}$ in $X$ such that $u_m\to x$ and $e^{\langle
z_m,A\rangle}u_m\to y$. By Lemma~\ref{kerim0}, $\kappa(n,A)\times
\{0\}\subseteq \Sigma$ and $\{0\}\times \kappa(n,A)\subseteq \Sigma$
for any $n\in\N^k$, where $\kappa(n,A)$ is defined in (\ref{KERk}).
On the other hand, $\Sigma$ is a linear subspace of $X\times X$.
Thus
$$
\smash{ \KER(A)\times\KER(A)=\spann\bigcup_{n\in\N^k}
\bigl((\kappa(n,A)\times \{0\})\cup (\{0\}\times
\kappa(n,A))\bigr)\subseteq\Sigma.\qedhere}
$$
\end{proof}

\subsection{Mixing semigroups and extended backward shifts}

We start by proving Proposition~\ref{untr2}. Proposition~G is
Proposition~1 in \cite{ge1}, while Theorem~U can be found in
\cite[pp.~348--349]{ge1}.

\begin{thmd} Let $X$ be a topological space and ${\cal
F}=\{T_\alpha:\alpha\in A\}$ be a family of continuous maps from $X$
to $X$ such that $T_\alpha T_\beta=T_\beta T_\alpha$ and
$T_\alpha(X)$ is dense in $X$ for any $\alpha,\beta\in A$. Then the
set of universal elements for $\cal F$ is either empty or dense in
$X$.
\end{thmd}

\begin{thmdu} Let $X$ be a Baire topological space, $Y$ be a second countable
topological space and $\{T_a:a\in A\}$ be a family of continuous
maps from $X$ into $Y$. Then the set of universal elements for
$\{T_a:a\in A\}$ is dense in $X$ if and only if $\{(x,T_ax):x\in X,\
a\in A\}$ is dense in $X\times Y$.
\end{thmdu}

\begin{proof}[Proof of Proposition~$\ref{untr2}$]
Assume that $\{T_t\}_{t\in A}$ is hereditarily hypercyclic. That is,
$\{T_{t_n}:n\in\Z_+\}$ is universal for any sequence
$\{t_n\}_{n\in\Z_+}$ in $A$ satisfying $|t_n|\to\infty$. Applying
this to $t_n=nt$ with $t\in A$, $|t|>0$, we see that $T_t$ is
hypercyclic. Since any hypercyclic operator has dense range
\cite{ge1}, $T_t(X)$ is dense in $X$ if $|t|>0$. Assume that
$\{T_t\}_{t\in A}$ is non-mixing. Then there are non-empty open
subsets $U$ and $V$ of $X$ and a sequence $\{t_n\}_{n\in\Z_+}$ in
$A$ such that $|t_n|\to\infty$ and $|t_n|>0$, $T_{t_n}(U)\cap
V=\varnothing$ for each $n\in\Z_+$. Since $T_{t_n}$ have dense
ranges and commute, Proposition~G implies that the set $W$ of
universal elements for $\{T_{t_n}:n\in\Z_+\}$ is dense in $X$. Hence
we can pick $x\in W\cap U$. Since $x$ is universal for
$\{T_{t_n}:n\in\Z_+\}$, there is $n\in\Z_+$ for which $T_{t_n}x\in
V$. Hence $T_{t_n}x\in T_{t_n}(U)\cap V=\varnothing$. This
contradiction completes the proof of the first part of
Proposition~\ref{untr2}.

Next, assume that $X$ is Baire separable and metrizable,
$\{T_t\}_{t\in A}$ is mixing and $\{t_n\}_{n\in\Z_+}$ is a sequence
in $A$ such that $|t_n|\to\infty$. By definition of mixing, for any
non-empty open subsets $U$ and $V$ of $X$, $T_{t_n}(U)\cap
V\neq\varnothing$ for all sufficiently large $n\in\Z_+$. Hence
$\{(x,T_{t_n}x):x\in X,\ n\in\Z_+\}$ is dense in $X\times X$. By
Theorem~U, $\{T_{t_n}:n\in\Z_+\}$ is universal.
\end{proof}

\begin{proposition}\label{kerim2}
Let $X$ be a topological vector space and $A=(A_1,\dots,A_k)\in
L(X)^k$ be a EBS$_k$-tuple such that $e^{\langle z,A\rangle}$ is
well-defined for $z\in \K^k$ and $\{e^{\langle
z,A\rangle}\}_{z\in\K^k}$ is an operator group. Then $\{e^{\langle
z,A\rangle}\}_{z\in\K^k}$ is mixing.
\end{proposition}

\begin{proof} Assume the contrary. Then we can find non-empty open
subsets $U$ and $V$ of $X$ and a sequence $\{z_m\}_{m\in\Z_+}$ in
$\K^k$ such that $|z_m|\to\infty$ and $e^{\langle
z_m,A\rangle}(U)\cap V=\varnothing$ for each $m\in\Z_+$. Let
$\Sigma$ be the set of $(x,y)\in X^2$ for which there is a sequence
$\{x_m\}_{m\in\Z_+}$ in $X$ such that $x_m\to x$ and $e^{\langle
z_m,A\rangle}x_m\to y$. By Corollary~\ref{kerim1}, $\KER(A)\times
\KER(A)\subseteq \Sigma$. Since $A$ is a EBS$_k$-tuple, $\KER(A)$ is
dense in $X$ and therefore $\Sigma$ is dense in $X\times X$. In
particular, $\Sigma$ meets $U\times V$, which is not possible since
$e^{\langle z_m,A\rangle}(U)\cap V=\varnothing$ for any $m\in\Z_+$.
This contradiction shows that $\{e^{\langle
z,A\rangle}\}_{z\in\K^k}$ is mixing.
\end{proof}

\begin{theorem}\label{mi00} Let $X$ be a separable Banach space and
$(A_1,\dots,A_k)\in L(X)^k$ be a EBS$_k$-tuple. Then $\{e^{\langle
z,A\rangle}\}_{z\in\K^k}$ is a hereditarily hypercyclic uniformly
continuous operator group on $X$.
\end{theorem}

\begin{proof}Since $A_j$ are pairwise commuting and $X$ is a Banach space,
$\{e^{\langle z,A\rangle}\}_{z\in\K^k}$ is a uniformly continuous
operator group. By Proposition~\ref{untr2}, it suffices to verify
that $\{e^{\langle z,A\rangle}\}_{z\in\K^k}$ is mixing. It remains
to apply Proposition~\ref{kerim2}.
\end{proof}

We will extend the above theorem to more general topological vector
spaces. Recall that a subset $A$ of a vector space is called {\it
balanced} if $zx\in A$ whenever $x\in A$, $z\in\K$ and $|z|\leq 1$.
A subset $D$ of a topological vector space $X$ is called a {\it
disk} if $D$ is convex, balanced and bounded. For a disk $D$, the
space $X_D=\spann(D)$ is endowed with the norm, being the Minkowskii
functional \cite{shifer} of $D$. Boundedness of $D$ implies that the
norm topology of $X_D$ is stronger than the topology inherited from
$X$. $D$ is called a {\it Banach disk} if the normed space $X_D$ is
complete. It is well-known \cite{bonet} that a compact disk is a
Banach disk.

\begin{lemma}\label{abcd} Let $X$ be a topological vector space, $p$
be a continuous seminorm on $X$, $D\subset X$ be a Banach disk, $q$
be the norm of $X_D$, $k\in\N$ and $A\in L(X)^k$ be a $k$-tuple of
pairwise commuting operators. Assume also that $A_j(X)\subseteq X_D$
for $1\leq j\leq k$ and there is $a>0$ such that $q(A_jx)\leq ap(x)$
for any $x\in X$ and $1\leq j\leq k$. Then for each $z\in\K^k$,
$e^{\langle z,A\rangle}$ is well-defined. Moreover, $\{e^{\langle
z,A\rangle}\}_{z\in\K^k}$ is a uniformly continuous operator group
and the map $z\mapsto f(e^{\langle z,A\rangle}x)$ from $\K^k$ to
$\K$ is analytic for any $x\in X$ and $f\in X'$. Furthermore, if
$X_D$ is separable and dense in $X$ and $B$ is an EBS$_k$-tuple,
then $\{e^{\langle z,A\rangle}\}_{z\in\K^k}$ is hereditarily
hypercyclic, where $B_j\in L(X_D)$ are restrictions of $A_j$ to
$X_D$.
\end{lemma}

\begin{proof}Since $D$ is bounded, there is $c>0$ such that $p(x)\leq
cq(x)$ for each $x\in X_D$. Since $q(A_jx)\leq ap(x)$ for each $x\in
X$, we have $q(A_jA_lx)\leq ap(A_lx)\leq caq(A_lx)\leq ca^2p(x)$.
Iterating this argument, we see that
\begin{equation}\label{estiii}
q(A_1^{n_1}\dots A_k^{n_k}x)\leq c^{|n|-1}a^{|n|} p(x)\ \ \
\text{for any $x\in X$ and $n\in\Z_+^k$, $|n|>0$,}
\end{equation}
where $|n|=n_1+{\dots}+n_k$. By (\ref{estiii}), for each $x\in X$
and $z\in\K^k$, the series
\begin{equation}\label{seri}
\llll\smash{\sum_{n\in\Z_+^k,\ |n|>0} \frac{z_1^{n_1}\dots
z_k^{n_k}}{n_1!\dots n_k!}A_1^{n_1}\dots A_k^{n_k}x}
\end{equation}
converges absolutely in the Banach space $X_D$. Since the series
$\sum\limits_{m=1}^\infty \frac{1}{m!}\langle z,A\rangle^mx$ can be
obtained from (\ref{seri}) by an appropriate 'bracketing', it is
also absolutely convergent in $X_D$. Hence the last series converges
in $X$ and therefore the formula $e^{\langle
z,A\rangle}x=\sum\limits_{m=0}^\infty \frac{1}{n!}\langle
z,A\rangle^mx$ defines a linear operator on $X$. Next, representing
$e^{\langle z,A\rangle}x-x$ by the series (\ref{seri}) and using
(\ref{estiii}), we obtain
$$
\llll\smash{q(e^{\langle z,A\rangle}x-x)\leq
\frac{p(x)}{c}\sum_{n\in\Z_+^k,\ |n|>0} \frac{|z_1|^{n_1}\dots
|z_k|^{n_k}}{n_1!\dots
n_k!}(ac)^{|n|}=\frac{p(x)}{c}(e^{ac\|z\|}-1),}
$$
where $\|z\|=|z_1|+{\dots}+|z_k|$. By the above display, each
$e^{\langle z,A\rangle}$ is continuous and $\{e^{\langle
z,A\rangle}\}_{z\in\K^k}$ is uniformly continuous. The semigroup
property follows in a standard way from the fact that $A_j$ are
pairwise commuting. Applying $f\in X'$ to the series (\ref{seri})
and using (\ref{estiii}), one immediately obtains the power series
expansion of the map $z\mapsto f(e^{\langle z,A\rangle}x)$. Hence
each $z\mapsto f(e^{\langle z,A\rangle}x)$ is analytic.

Assume now that $X_D$ is separable and dense in $X$, $B_j\in L(X_D)$
are restrictions of $A_j$ to $X_D$ and $B=(B_1,\dots,B_k)$ is an
EBS$_k$-tuple. By Theorem~\ref{mi00}, $\{e^{\langle
z,B\rangle}\}_{z\in\K^k}$ is hereditarily hypercyclic. Since each
$e^{\langle z,B\rangle}$ is the restriction of $e^{\langle
z,A\rangle}$ to $X_D$, $X_D$ is dense in $X$ and carries a topology
stronger than the one inherited from $X$, $\{e^{\langle
z,A\rangle}\}_{z\in\K^k}$ is also hereditarily hypercyclic.
\end{proof}

\section{$\ell_1$-sequences, equicontinuous sets and the class $\M$}

\begin{definition}\label{ELL1}We say that a sequence $\{x_n\}_{n\in\Z_+}$
in a topological vector space $X$ is an $\ell_1$-{\it sequence} if
the series $\sum\limits_{n=0}^\infty a_nx_n$ converges in $X$ for
each $a\in\ell_1$ and for any neighborhood $U$ of $0$ in $X$, there
is $n\in\Z_+$ such that $D_n\subseteq U$, where
$D_n=\Bigl\{\sum\limits_{k=0}^\infty a_kx_{n+k}:a\in\ell_1,\
\|a\|\leq 1\Bigr\}$.
\end{definition}

If $X$ is a locally convex space, the latter condition is satisfied
if and only if $x_n\to 0$.

\begin{lemma}\label{l11} Let $\{x_n\}_{n\in\Z_+}$ be an
$\ell_1$-sequence in a topological vector space $X$. Then the closed
balanced convex hull $D$ of $\{x_n:n\in\Z_+\}$ is compact and
metrizable. Moreover, $D=D'$, where
$D'=\Bigl\{\sum\limits_{n=0}^\infty a_nx_n:a\in\ell_1,\ \|a\|_1\leq
1\Bigr\}$, $X_D$ is separable and $E=\spann\{x_n:n\in \Z_+\}$ is
dense in the Banach space $X_D$.
\end{lemma}

\begin{proof} Let $Q=\{a\in \ell_1:\|a\|_1\leq 1\}$ be endowed with the
coordinatewise convergence topology. Then $Q$ is a metrizable and
compact as a closed subspace of $\D^{\Z_+}$, where
$\D=\{z\in\K:|z|\leq 1\}$. Obviously, the map $\Phi:Q\to D'$,
$\Phi(a)=\sum\limits_{n=0}^\infty a_nx_n$ is onto. Using the
definition of an $\ell_1$-sequence, one can in a routine way verify
that $\Phi$ is continuous. Hence $D'$ is compact and metrizable as a
continuous image of a compact metrizable space. Thus $D'$, being
also balanced and convex, is a Banach disk. Let $u\in X_{D'}$ and
$a\in \ell_1$ be such that $u=\Phi(a)$. One can easily see that
$\ssub{p}{D'}(u_n-u)\to 0$, where $u_n=\sum\limits_{k=0}^n a_kx_k$.
Hence $u_n\to u$ in $X$. Moreover, if $u\in D'$, then $u_n$ are in
the balanced convex hull of $\{x_n\}_{n\in\Z_+}$. Thus $D$ is dense
and closed in $D'$ and therefore $D=D'$. Hence
$\ssub{p}{D}(u_n-u)\to 0$ for each $u\in X_D$. Since $u_n\in E$, $E$
is dense in $X_D$ and $X_D$ is separable.
\end{proof}

\begin{lemma}\label{mmmm}Let $X$ be a topological vector space. Then
the following are equivalent$:$
\begin{itemize}\itemsep=-2pt
\item[\rm (\ref{mmmm}.1)] $X\in \M_0;$
\item[\rm (\ref{mmmm}.2)] there exists a Banach disk $D$ in $X$
with dense linear span such that $X_D$ is separable$;$
\item[\rm (\ref{mmmm}.3)] there exists an $\ell_1$-sequence in $X$ with
dense linear span.
\end{itemize}
\end{lemma}

\begin{proof} Obviously, (\ref{mmmm}.2) implies
(\ref{mmmm}.1). Lemma~\ref{l11} ensures that (\ref{mmmm}.3)  implies
(\ref{mmmm}.2). It remains to verify that (\ref{mmmm}.1) implies
(\ref{mmmm}.3). Assume that $X\in \M_0$. Then there is a dense
linear subspace $Y$ of $X$ carrying its own topology $\tau$ stronger
than the topology inherited from $X$ such that $Y=(Y\tau)$ is a
separable $\F$-space. Clearly any $\ell_1$-sequence in $Y$ with
dense linear span is also an $\ell_1$ sequence in $X$ with dense
linear span. Thus it suffices to find an $\ell_1$-sequence with
dense linear span in $Y$. To this end, we pick a dense subset
$A=\{y_n:n\in\Z_+\}$ of $Y$ and a base $\{U_n\}_{n\in\Z_+}$ of
neighborhoods of $0$ in $Y$ such that each $U_n$ is balanced and
$U_{n+1}+U_{n+1}\subseteq U_n$ for $n\in\Z_+$. Pick a sequence
$\{c_n\}_{n\in\Z_+}$ of positive numbers such that $x_n=c_ny_n\in
U_n$ for each $n\in\Z_+$. It is now easy to demonstrate that
$\{x_n\}_{n\in\Z_+}$ is an $\ell_1$-sequence in $Y$ with dense span.
\end{proof}

\begin{lemma}\label{bp1} Let $X$ be a separable metrizable
topological vector space and $\{f_n\}_{n\in\Z_+}$ be a linearly
independent sequence in $X'$. Then there exist sequences
$\{x_n\}_{n\in\Z_+}$ in $X$ and $\{\alpha_{k,j}\}_{k,j\in\Z_+,\
j<k}$ in $\K$ such that $\spann\{x_k:k\in\Z_+\}$ is dense in $X$,
$g_n(x_k)=0$ for $n\neq k$ and $g_n(x_n)\neq 0$ for $n\in\Z_+$,
where $g_n=f_n+\sum\limits_{j<n}\alpha_{n,j}f_j$.
\end{lemma}

\begin{proof} Let $\{U_n\}_{n\in\Z_+}$ be a base of topology of $X$.
We construct inductively sequences $\{\alpha_{k,j}\}_{k,j\in\Z_+,\
j<k}$ in $\K$ and $\{y_n\}_{n\in\Z_+}$ in $X$ such that for any
$k\in\Z_+$,
\begin{equation}\label{b1-3}
\text{$y_k\in U_k$,\ \ \ $g_k(y_k)\neq 0$\ \ \ and\ \ \ $g_k(y_m)=0$
if $m<k$, where \ $g_k=f_k+\sum\limits_{j<k}\alpha_{k,j}f_j$.}
\end{equation}
Let $g_0=f_0$. Since $f_0\neq 0$, there is $y_0\in U_0$ such that
$f_0(y_0)=g_0(y_0)\neq 0$. This provides us with the base of
induction. Assume now that $n\in\N$ and $y_k$, $\alpha_{k,j}$ with
$j<k<n$ satisfying (\ref{b1-3}) are already constructed. According
to (\ref{b1-3}), we can find $\alpha_{n,0},\dots,\alpha_{n,n-1}\in
\K$ such that $g_n(y_m)=0$ for $m<n$, where
$g_n=f_n+\sum\limits_{j<n}\alpha_{n,j}f_j$. Since $f_j$ are linearly
independent, $g_n\neq 0$ and therefore there is $y_n\in U_n$ such
that $g_n(y_n)\neq 0$. This concludes the inductive procedure.

Using (\ref{b1-3}), one can choose a sequence
$\{\beta_{k,j}\}_{k,j\in\Z_+,\ j<k}$ in $\K$ such that $g_n(x_n)\neq
0$ for $n\in\Z_+$ and $g_n(x_k)=0$ for $k\neq m$, where
$x_k=y_k+\sum\limits_{j<k}\beta_{k,j}y_j$. Since $y_n\in U_n$,
$\{y_n:n\in\Z_+\}$ is dense in $X$. Hence
$\spann\{x_n:n\in\Z_+\}=\spann\{y_n:n\in\Z_+\}$ is dense in $X$.
\end{proof}

\begin{lemma}\label{equi2}Let $X\in\M_1$. Then there exists a
linearly independent equicontinuous sequence $\{f_n:n\in\Z_+\}$ in
$X'$ such that $\phi\subseteq \bigl\{\{f_n(x)\}_{n\in\Z_+}: x\in
X\bigr\}$.
\end{lemma}

\begin{proof} Since $X\in\M_1$, there is a continuous seminorm $p$ on $X$
for which $X_p=X/\ker p$ with the norm $\|x+\ker p\|=p(x)$ is an
infinite dimensional normed space. Since every infinite dimensional
normed space admits a biorthogonal sequence, we can choose sequences
$\{x_n\}_{n\in\Z_+}$ in $X$ and $\{g_n\}_{n\in\Z_+}$ in $X'_p$ such
that $\|g_n\|\leq 1$ for each $n\in\Z_+$ and $g_n(x_k+\ker
p)=\delta_{n,k}$ for $n,k\in\Z_+$, where $\delta_{n,k}$ is the
Kronecker delta. Now let $f_n:X\to \K$, $f_n(x)=g_n(x+\ker p)$. The
above properties of $g_n$ can be rewritten in terms of $f_n$ in the
following way: $|f_n(x)|\leq p(x)$ and $f_n(x_k)=\delta_{n,k}$ for
any $n,k\in\Z_+$ and $x\in Y$. Since $f_n(x_k)=\delta_{n,k}$, we
have $\phi\subseteq \bigl\{\{f_n(x)\}_{n\in\Z_+}: x\in X\bigr\}$. By
the inequality $|f_n(x)|\leq p(x)$, $\{f_n:n\in\Z_+\}$ is
equicontinuous.
\end{proof}

\begin{lemma}\label{ma1} Let $X\in \M$. Then there exist an $\ell_1$-sequence
$\{x_n\}_{n\in\Z_+}$ in $X$ with dense linear span and an
equicontinuous sequence $\{f_k\}_{k\in\Z_+}$ in $X'$ such that
$f_k(x_n)=0$ if $k\neq n$ and $f_k(x_k)\neq0$ for each $k\in\Z_+$.
\end{lemma}

\begin{proof}According to Lemma~\ref{mmmm}, there is a Banach
disk $D$ in $X$ such that $X_D$ is separable and dense in $X$. By
Lemma~\ref{equi2}, there is a linearly independent equicontinuous
sequence $\{g_n\}_{n\in\N}$ in $X'$. Since $X_D$ is dense in $X$,
the functionals $g_n\bigr|_{X_D}$ on $X_D$ are linearly independent.
Applying Lemma~\ref{bp1} to the sequence $\{g_n\bigr|_{X_D}\}$, we
find sequences $\{y_n\}_{n\in\Z_+}$ in $X_D$ and
$\{\alpha_{k,j}\}_{k,j\in\Z_+,\ j<k}$ in $\K$ such that
$E=\spann\{y_k:k\in\Z_+\}$ is dense in $X_D$, $h_n(y_k)=0$ for
$n\neq k$ and $h_n(y_n)\neq 0$ for $n\in\Z_+$, where
$h_n=g_n+\sum\limits_{j<n}\alpha_{n,j}g_j$. Consider $f_n=c_nh_n$,
where $c_n=\Bigl(1+\sum\limits_{j<n}|\alpha_{n,j}|\Bigr)^{-1}$.
Since $\{g_n:n\in\N\}$ is equicontinuous, $\{f_n:n\in\N\}$ is also
equicontinuous. Next, let $x_n=b_ny_n$, where
$b_n=2^{-n}q(x_n)^{-1}$ and $q$ is the norm of the Banach space
$X_D$. Since $x_n$ converges to 0 in $X_D$, $\{x_n\}_{n\in\N}$ is an
$\ell_1$-sequence in $X_D$. Since $X_D$ is dense in $X$,
$\spann\{x_n:n\in\Z_+\}=E$ is dense in $X_D$, and the topology of
$X_D$ is stronger than the one inherited from $X$,
$\{x_n\}_{n\in\N}$ is an $\ell_1$-sequence in $X$ with dense linear
span. Finally since $f_n(x_k)=c_nb_kh_n(y_k)$, we see that
$f_n(x_k)=0$ if $n\neq k$ and $f_n(x_n)\neq 0$ for any $n\in\Z_+$.
Thus all required conditions are satisfied.
\end{proof}

\subsection{Proof of Proposition~\ref{fsp}}

Let $X$ be a separable $\F$-space. We have to show that $X$ belongs
to $\M$ if and only if $\dim X'>\aleph_0$.

First, assume that $X\in\M$. Then there is a continuous seminorm $p$
on $X$ such that $X_p=X/\ker p$ is infinite dimensional. We endow
$X_p$ with the norm $\|x+\ker p\|=p(x)$. The dual $X'_p$ of the
normed space $X_p$ is naturally contained in $X'$. Since the
algebraic dimension of the dual of any infinite dimensional normed
space is at least $2^{\aleph_0}$ \cite{bonet}, we have $\dim
X'\geq\dim X'_p\geq 2^{\aleph_0}>\aleph_0$.

Assume now that $\dim X'>\aleph_0$ and let $\{U_n\}_{n\in\Z_+}$ be a
base of neighborhoods of $0$ in $X$. Then $X'$ is the union of
subspaces $Y_n=\{f\in X':|f|\ \ \text{is bounded on}\ \ U_n\}$ for
$n\in\Z_+$. Since $\dim X'>\aleph_0$, we can pick $n\in\Z_+$ such
that $Y_n$ is infinite dimensional. Now let $p$ be the Minkowskii
functional of $U_n$. Then the open unit ball of $p$ is exactly the
balanced convex hull $W$ of $U_n$. Since $U_n\subseteq W$, $p$ is a
continuous seminorm on $X$. Since each $f\in Y_n$ is bounded on $W$
and $Y_n$ is infinite dimensional, $X/\ker p$  is also infinite
dimensional. Hence $X\in\M_1$. Since $X$, as a separable $\F$-space,
belongs to $\M_0$, we see that $X\in\M$. The proof is complete.

\section{Proof of Theorem~\ref{saan}}

Let $X\in\M$. By Lemma~\ref{ma1}, there exist an $\ell_1$-sequence
$\{x_n\}_{n\in\Z_+}$ in $X$ and an equicontinuous sequence
$\{f_k\}_{k\in\Z_+}$ in $X'$ such that $E=\spann\{x_n:n\in\Z_+\}$ is
dense in $X$, $f_k(x_n)=0$ if $k\neq n$ and $f_k(x_k)\neq0$ for each
$k\in\Z_+$. Since $\{f_k\}$ is equicontinuous, there is a continuous
seminorm $p$ on $X$ such that each $|f_k|$ is bounded by $1$ on the
unit ball of $p$. Since $\{x_n\}$ is an $\ell_1$-sequence in $X$,
Lemma~\ref{l11} implies that the balanced convex closed hull $D$ of
$\{x_n:n\in\Z_+\}$ is a Banach disk in $X$. Let $q$ be the norm of
the Banach space $X_D$. Then $q(x_n)\leq 1$ for each $n\in\Z_+$.

\begin{lemma}\label{ex} Let $\alpha,\beta:\Z_+\to\Z_+$ be any maps
and $a=\{a_n\}_{n\in\Z_+}\in \ell_1$. Then the formula
\begin{equation}\label{T}
Tx=\sum_{n\in\Z_+} a_n f_{\alpha(n)}(x)x_{\beta(n)}
\end{equation}
defines a continuous linear operator on $X$. Moreover,
$T(X)\subseteq X_D$ and $q(Tx)\leq \|a\|p(x)$ for each $x\in X$,
where $\|a\|$ is the $\ell_1$-norm of $a$.
\end{lemma}

\begin{proof} Since $\{f_k\}$ is equicontinuous, $\{f_{\alpha(n)}(x)\}_{n\in\Z_+}$
is bounded for any $x\in X$. Since $\{x_n\}$ is an $\ell_1$-sequence
and $a\in\ell_1$, the series in (\ref{T}) converges for any $x\in X$
and therefore defines a linear operator on $X$. Moreover, if
$p(x)\leq 1$, then $|f_k(x)|\leq 1$ for each $k\in\Z_+$. Since
$q(x_m)\leq 1$ for $m\in\Z_+$, (\ref{T}) implies that $q(Tx)\leq
\|a\|$ if $p(x)\leq 1$. Hence $q(Tx)\leq \|a\|p(x)$ for each $x\in
X$. It follows that $T$ is continuous and takes values in $X_D$.
\end{proof}

Fix a bijection $\gamma:\Z_+^k\to\Z_+$. By $e_j$ we denote the
element of $\Z_+^k$ defined by $(e_j)_l=\delta_{j,l}$. For
$n\in\Z_+^k$, we write $|n|=n_1+{\dots}+n_k$. Let
$$
\epsilon_m=\min\bigl\{\bigl|f_{\gamma(n)}(x_{\gamma(n)})\bigr|:n\in\Z_+^k,\
|n|=m+1\bigr\}\qquad \text{for $m\in\Z_+$.}
$$
Since $f_j(x_j)\neq 0$, $\epsilon_m>0$ for $m\in\Z_+$. Pick any
sequence $\{\alpha_m\}_{m\in\Z_+}$ of positive numbers satisfying
\begin{equation}\label{al}
\alpha_{m+1}\geq 2^m\alpha_m\epsilon_m^{-1}\quad\text{for any
$m\in\Z_+$}
\end{equation}
and consider the operators $A_j:X\to X$ defined by the formula
$$
A_jx=\sum_{n\in\Z_+^k} \frac{\alpha_{|n|}f_{\gamma(n+e_j)}(x)}
{\alpha_{|n|+1}f_{\gamma(n+e_j)}(x_{\gamma(n+e_j)})} x_{\gamma(n)}\
\ \ \text{for $1\leq j\leq k$}.
$$
By (\ref{al}), the series defining $A_j$ can be written as
\begin{equation*}
A_jx=\sum_{n\in\Z_+^k} c_{j,n}f_{\gamma(n+e_j)}(x) x_{\gamma(n)}\ \
\text{with $0<|c_{j,n}|<2^{-|n|}$}\text{\ \ and therefore\ }
\sum_{n\in\Z_+^k}|c_{j,n}|\leq C=\!\!\sum_{n\in\Z_+^k} 2^{-|n|}.
\end{equation*}
Then each $A_j$ has shape (\ref{T}) with $\|a\|\leq C$. By
Lemma~\ref{ex}, $A_j\in L(X)$, $A_j(X)\subseteq X_D$ and $q(Tx)\leq
Cp(x)$ for any $x\in X$. Using the definition of $A_j$ and the
equalities $f_m(x_j)=0$ for $m\neq j$, it is easy to verify that
$A_jA_lx_n=A_lA_jx_n$ for any $1\leq j<l\leq k$ and $n\in\Z_+$.
Indeed, for any $n\in\Z_+$, there is a unique $m\in\Z_+^k$ such that
$n=\gamma(m)$. If either $m_j=0$ or $m_l=0$, we have
$A_jA_lx_n=A_lA_jx_n=0$. If $m_j\geq 1$ and $m_l\geq 1$, then
$A_jA_lx_n=A_lA_jx_n=\frac{\alpha_{|m|-2}}{\alpha_{|m|}}x_{\gamma(m-e_j-e_l)}$.
Since $E$ is dense in $X$, $A_1,\dots,A_n$ are pairwise commuting.
By Lemma~\ref{abcd}, $e^{\langle z,A\rangle}$ are well-defined for
$z\in\K^k$, $\{e^{\langle z,A\rangle}\}_{z\in\K^k}$ is a uniformly
continuous operator group and the map $z\mapsto f(e^{\langle
z,A\rangle}x)$ from $\K^k$ to $\K$ is analytic for any $x\in X$ and
$f\in X'$. It remains to show that $\{e^{\langle
z,A\rangle}\}_{z\in\K^k}$ is hereditarily hypercyclic. By
Lemma~\ref{l11}, $X_D$ is separable. According to Lemma~\ref{abcd},
it suffices to prove that $B\in L(X_D)^k$ is an EBS$_k$-tuple, where
$B_j$ are restrictions of $A_j$ to $X_D$. Clearly $B_j$ commute as
restrictions of commuting operators. Using the relations
$f_m(x_j)=0$ for $m\neq j$ and $f_j(x_j)\neq 0$, it is easy to see
that the set $\kappa(m,B)$, defined in (\ref{KERk}), contains
$E_m=\spann\{x_{\gamma(n)}:n\in\Z_+^k,\ n_j\leq m_j-1,\ 1\leq j\leq
k\}$ for each $m\in\N^k$. Hence $\KER B$, defined in (\ref{KERk}),
contains $E$, which is dense in $X_D$ by Lemma~\ref{l11}. Thus $B$
is an EBS$_k$-tuple. The proof of Theorem~\ref{saan} is complete.

\section{Spaces without supercyclic semigroups $\{T_t\}_{t\in\R_+}$}

\begin{lemma}\label{fds}Let $X$ be a finite dimensional topological
vector space of the $\R$-dimension $>2$. Then there is no
supercyclic strongly continuous operator semigroup
$\{T_t\}_{t\in\R_+}$ on $X$.
\end{lemma}

\begin{proof} As well-known, any strongly continuous operator
semigroup $\{T_t\}_{t\in\R_+}$ on $\K^n$ has shape
$\{e^{tA}\}_{t\in\R_+}$, where $A\in L(\K^n)$. Assume the contrary.
Then there exist $n\in\N$ and $A\in L(\K^n)$ such that
$\{e^{tA}\}_{t\in\R_+}$ is supercyclic and $\dimr \K^n>2$. Since
$e^{tA}$ are invertible and commute with each other, Proposition~G
implies that the set $W$ of universal elements for
$\{ze^{tA}:z\in\K,\ t\in\R_+\}$ is dense in $\K^n$. On the other
hand, for each $c>0$ and any $x\in\K^n$, from the restriction on $n$
it follows that the closed set $\{ze^{tA}x:z\in\K,\ 0\leq t\leq c\}$
is nowhere dense in $\K^n$ (smoothness of the map $(z,t)\mapsto
ze^{tA}x$ implies that the topological dimension of
$\{ze^{tA}x:z\in\K,\ 0\leq t\leq c\}$ is less than that of $\K^n$).
Hence, each $x\in W$ is universal for $\{ze^{tA}:z\in\K,\ t>c\}$ for
any $c>0$. Now if $(a,b)$ is a subinterval of $(0,\infty)$, it is
easy to see that the family $\{ze^{tkA}:z\in\K,\ a<t<b,\ k\in\Z_+\}$
contains $\{ze^{tA}:z\in\K,\ t>c\}$ for a sufficiently large $c>0$.
Hence for each $x\in W$, the set $\{ze^{tkA}x:z\in\K,\ a<t<b,\
k\in\Z_+\}$ is dense in $\K^n$. Since $(a,b)$ is arbitrary and $W$
is dense in $\K^n$, $\{(t,x,ze^{tkA}x:t\in\R_+,\ z\in\K,\ x\in\K^n,
k\in\Z_+\}$ is dense in $\R_+\times\K^n\times\K^n$. By Theorem~U,
the family $\{F_{z,k}:z\in\K,\ k\in\Z_+\}$ of maps
$F_{z,k}:\R_+\times \K^n\to\K^n$, $F_{z,k}(t,x)=ze^{tkA}x$ has dense
set $U_0\subset\R_+\times\K^n$ of universal elements. Hence the
projection $U$ of $U_0$ onto $\K^n$ is  dense in $\K^n$. On the
other hand, $U$ is exactly the set of $x\in\K^n$ supercyclic for
$e^{tA}$ for some $t\in\R_+$. In particular, there is $t\in\R_+$
such that $e^{tA}$ is supercyclic. This contradicts the fact (see
\cite{ww}) that there are no supercyclic operators on finite
dimensional spaces of real dimension $>2$.
\end{proof}

\begin{remark}\label{r1} In the proof of Lemma~\ref{fds} we have
shown that a strongly continuous  supercyclic operator semigroup on
a finite dimensional space must contain supercyclic operators. It is
worth mentioning that Conejero, M\"uller and Peris \cite{comupe}
proved that every $T_t$ with $t>0$ is hypercyclic for any
hypercyclic strongly continuous operator semigroup
$\{T_t\}_{t\in\R_+}$ on an $\cal F$-space. Bernal-Gonz\'alez and
Grosse-Erdmann \cite{ex2} gave an example of a supercyclic strongly
continuous operator semigroup $\{T_t\}_{t\in\R_+}$ on a real Hilbert
space such that $T_t$ is not supercyclic for $t$ from a dense subset
of $\R_+$. It seems to remain unknown whether $T_t$ with $t>0$ must
all be supercyclic for every supercyclic strongly continuous
operator semigroup $\{T_t\}_{t\in\R_+}$ on a {\it complex} $\cal
F$-space.
\end{remark}

The following (trivial under the Continuum Hypothesis) result is
Lemma~2 in \cite{shka1}.

\begin{lemma}\label{ch} Let $(M,d)$ be a separable complete metric
space, $X$ be a topological vector space,\break $f:M\to X$ be a
continuous map and $\tau=\dim\,\spann f(M)$. Then either
$\tau\leq\aleph_0$ or $\tau=2^{\aleph_0}$.
\end{lemma}

\begin{lemma}\label{dich} Let $\{T_t\}_{t\in\R_+}$ be a strongly
continuous operator semigroup on a topological vector space $X$,
$x\in X$ and $C(x)=\spann\{T_tx:t\in\R_+\}$. Then either $\dim
C(x)<\aleph_0$ or $\dim C(x)=2^{\aleph_0}$.
\end{lemma}

\begin{proof} By Lemma~\ref{ch}, either $\dim C(x)\leq \aleph_0$
or $\dim C(x)=2^{\aleph_0}$. It remains to rule out the case $\dim
C(x)=\aleph_0$. Assume that $\dim C(x)=\aleph_0$. Restricting the
$T_t$ to the invariant subspace $C(x)$, we can without loss of
generality assume that $C(x)=X$. Thus $\dim X=\aleph_0$ and
therefore $X$ is the union of an increasing sequence
$\{X_n\}_{n\in\Z_+}$ of finite dimensional subspaces. First, we
shall show that for each $\epsilon>0$, the space
$X_\epsilon=\spann\{T_tx:t\geq\epsilon\}$ is finite dimensional.

Let $\epsilon>0$ and $0<\alpha<\epsilon$. Then $[\alpha,\epsilon]$
is the union of closed sets $A_n=\{t\in[\alpha,\epsilon]:T_t x\in
X_n\}$ for $n\in\Z_+$. By the Baire category theorem, there is
$n\in\Z_+$ such that $A_n$ has non-empty interior in
$[\alpha,\epsilon]$. Hence we can pick $a,b\in\R$ such that
$\alpha\leq a<b\leq\epsilon$ and $T_tx\in X_n$ for any $t\in [a,b]$.
We shall show that $T_tx\in X_n$ for $t\geq a$. Assume, it is not
the case. Then the number $c=\inf\{t\in[a,\infty):T_tx\notin X_n\}$
belongs to $[b,\infty)$. Since $\{t\in\R_+:T_t\in X_n\}$ is closed,
$T_cx\in X_n$. Since $[a,b]$ is uncountable and the span of
$\{T_t:t\in[a,b]\}$ is finite dimensional, we can pick $a\leq
t_0<t_1<{\dots}<t_n\leq b$ and $c_1,\dots,c_{n-1}\in\K$ such that
$T_{t_n}x=c_1T_{t_1}x+{\dots}+c_{n-1}T_{t_{n-1}}x$. Since $T_cx\in
X_n$, by definition of $c$, there is $t\in (c,c+t_n-t_{n-1})$ such
that $T_tx\notin X_n$. Since $t>c\geq t_n$, the equality
$T_{t_n}x=c_1T_{t_1}x+{\dots}+c_{n-1}T_{t_{n-1}}x$ implies that
$T_tx=T_{t-t_n}T_{t_n}x=T_{t-t_n}\sum\limits_{j=1}^{n-1}
c_jT_{t_j}x=\sum\limits_{j=1}^{n-1} c_jT_{t-t_n+t_j}x\in X_n$
because $a\leq t-t_n+t_j\leq c$ for $1\leq j\leq n-1$. This
contradiction proves that $T_tx\in X_n$ for each $t\geq a$. Hence
$X_\epsilon\subseteq X_n$ and therefore $X_\epsilon$ is finite
dimensional for each $\epsilon>0$.

Since $T_t(X)=T_t(C(x))\subseteq X_t$, $T_t$ has finite rank for any
$t>0$. Let $t>0$. Since $T_t$ has finite rank, $F_t=\ker T_t$ is a
closed subspace of $X$ of finite codimension. Clearly $F_t$ is
$T_s$-invariant for each $s\in\R_+$. Passing to quotient operators,
$S_s\in L(X/F_t)$, $S_s(u+F_t)=T_su+F_t$, we get a strongly
continuous semigroup $\{S_s\}_{s\in\R_+}$ on the finite dimensional
space $X/F_t$. Hence there is $A\in L(X/F_t)$ such that $S_s=e^{sA}$
for $s\in\R_+$. Thus each $S_s$ is invertible and is a quotient of
$T_s$, we obtain ${\tt rk}\,T_s\geq {\tt rk}\,S_s=\dim X/F_t={\tt
rk}\,T_t$ for any $t>0$ and $s\geq 0$. Thus $T_t$ for $t>0$ have the
same rank $k\in\N$. Passing to the limit as $t\to 0$, we see that
the identity operator $I=T_0$ is the strong operator topology limit
of a sequence of rank $k$ operators. Hence ${\tt rk}\,I\leq k$. That
is, $X$ is finite dimensional. This contradiction completes the
proof.
\end{proof}

\begin{lemma}\label{phi} Let $X$ be a topological vector space
in which the linear span of each metrizable compact subset has
dimension $<2^{\aleph_0}$. Then for any strongly continuous operator
semigroup $\{T_t\}_{t\in\R_+}$ on $X$ and any $x\in X$, the space
$C(x)=\spann\{T_tx:t\in\R_+\}$ is finite dimensional.
\end{lemma}

\begin{proof} Let $\{T_t\}_{t\in\R_+}$ be a strongly continuous
operator semigroup on $X$ and $x\in X$. By strong continuity,
$K_n=\{T_tx:0\leq t\leq n\}$ is compact and metrizable for any
$n\in\N$. Hence $\dim E_n<2^{\aleph_0}$ for any $n\in\N$, where
$E_n=\spann(K_n)$. Since the sum of countably many cardinals
strictly less than $2^{\aleph_0}$ is strictly less than
$2^{\aleph_0}$, $\dim C(x)\leq \sum\limits_{n=1}^\infty \dim
E_n<2^{\aleph_0}$. By Lemma~\ref{dich}, $C(x)$ is finite
dimensional.
\end{proof}

Applying Lemma~\ref{fds} if $X$ is finite dimensional and
Lemma~\ref{phi} otherwise, we get the following result.

\begin{corollary}\label{first} Let $X$ be a topological vector space
such that $\dimr  X>2$ and the linear span of each metrizable
compact subset of $X$ has dimension $<2^{\aleph_0}$. Then there is
no strongly continuous supercyclic operator semigroup
$\{T_t\}_{t\in\R_+}$ on $X$.
\end{corollary}

\begin{corollary}\label{second} Let $X$ be an infinite dimensional
topological vector space such that $\dimr  X'>2$ and in $X'_\sigma$
the span of any compact metrizable subset has dimension
$<2^{\aleph_0}$. Then there is no strongly continuous supercyclic
operator semigroup $\{T_t\}_{t\in\R_+}$ on $X$.
\end{corollary}

\begin{proof} Assume that there exists a supercyclic strongly continuous
operator semigroup $\{T_t\}_{t\in\R_+}$ on $X$. It is
straightforward to verify that $\{T'_t\}_{t\in\R_+}$ is a strongly
continuous semigroup on $X'_\sigma$. Pick any finite dimensional
subspace $L$ of $X'$ such that $\dimr  L>2$. By Lemma~\ref{phi},
$E=\spann\{T'_t f:t\in\R_+,\ f\in L\}$ is finite dimensional. Since
$L\subseteq E$, $\dimr  E>2$. Since $E$ is $T'_t$-invariant for any
$t\in\R_+$, its annihilator $F=\{x\in X:f(x)=0\ \text{for any}\ f\in
E\}$ is $T_t$-invariant for each $t\in\R_+$. Thus we can consider
the quotient operators $S_t\in L(X/F)$, $S_t (x+F)=T_tx+F$. Then
$\{S_t\}_{t\in\R_+}$ is a strongly continuous operator semigroup on
$X/F$. Moreover, $\{S_t\}_{t\in\R_+}$ is supercyclic since
$\{T_t\}_{t\in\R_+}$ is. Now since $\dim E=\dim X/F$, $2<\dimr
X/F<\aleph_0$. By Lemma~\ref{fds}, there are no strongly continuous
supercyclic operator semigroups on $X/F$. This contradiction
completes the proof.
\end{proof}

\begin{proof}[Proof of Theorem~\ref{omegG}] Theorem~$\ref{omegG}$ immediately follows from
Corollaries~\ref{first} and~\ref{second}. \end{proof}

\section{Examples, remarks and questions}

Note that if $(X,\tau)\in\M$ is locally convex, then
$(X,\theta)\in\M$ for any locally convex topology $\theta$ on $X$
such that $\theta\neq\sigma(X,X')$ and $(X,\theta)$ has the same
dual $X'$ as $(X,\tau)$. This is an easy application of the
Mackey--Arens theorem \cite{shifer}. Moreover, if $(X,\tau)\in\M$ is
locally convex, the hereditarily hypercyclic uniformly continuous
group from Theorem~\ref{saan} is strongly continuous and
hereditarily hypercyclic on $X$ equipped with the weak topology.
Unfortunately, the nature of the weak topology does not allow to
make such a semigroup uniformly continuous.

Assume now that $X$ is an infinite dimensional separable $\F$-space.
If $\dim X'>\aleph_0$, Proposition~\ref{fsp} and Theorem~\ref{saan}
provide uniformly continuous hereditarily hypercyclic operator
groups $\{T_t\}_{t\in\K^k}$ on $X$. If $2<\dimr  X'\leq\aleph_0$,
Theorem~\ref{omegG} does not allow a supercyclic strongly continuous
operator semigroup $\{T_t\}_{t\in\R_+}$ on $X$. Similarly, if
$1\leq\dim X'\leq\aleph_0$, there are no hypercyclic strongly
continuous operator semigroups $\{T_t\}_{t\in\R_+}$ on $X$. It
leaves unexplored the case $X'=\{0\}$.

\begin{question}\label{fspac2} Characterize infinite dimensional
separable $\F$-spaces $X$ such that $X'=\{0\}$  and $X$ admits a
hypercyclic strongly continuous operator semigroup
$\{T_t\}_{t\in\R_+}$. In particular, is it true that an $\F$-space
$X$ with $X'=\{0\}$ supporting a hypercyclic operator, supports also
a hypercyclic strongly continuous operator semigroup
$\{T_t\}_{t\in\R_+}$?
\end{question}

Recall that an infinite dimensional topological vector space $X$ is
called {\it rigid} if $L(X)$ consists only of the operators of the
form $\lambda I$ for $\lambda\in\K$. Since there exist rigid
separable $\F$-spaces \cite{kal}, there are separable infinite
dimensional $\F$-spaces on which support no cyclic operators or
cyclic strongly continuous operator semigroups $\{T_t\}_{t\in\R_+}$.
Of course, $X'=\{0\}$ if $X$ is rigid. We show that the equality
$X'=\{0\}$ for an $\F$-space is not an obstacle for having uniformly
continuous hereditarily hypercyclic operator groups. The spaces we
consider are $L_p[0,1]$ for $0\leq p<1$.

Let $(\Omega,{\cal A},\mu)$ be a measure space with $\mu$ being
$\sigma$-finite. Recall that if $0<p<1$, then $L_p(\Omega,\mu)$
consists of (classes of equivalence up to being equal almost
everywhere with respect to $\mu$ of) measurable functions
$f:\Omega\to\K$ satisfying $q_p(f)=\int_\Omega
|f(x)|^p\,\lambda(dx)<\infty$ with the topology defined by the
metric $d_p(f,g)=q_p(f-g)$. The space $L_0(\Omega,\mu)$ consists of
(equivalence classes of) all measurable functions $f:\Omega\to\K$
with the topology defined by the metric $d_0(f,g)=q_0(f-g)$, where
$q_0(h)=\sum\limits_{n=0}^\infty\frac{2^{-n}}{\mu(\Omega_n)}\int_{\Omega_n}
\frac{|f(x)|}{1+|f(x)|}\,\mu(dx)$ and $\{\Omega_n\}_{n\in\Z_+}$ is a
sequence of measurable subsets of $\Omega$ such that
$\mu(\Omega_n)<\infty$ for each $n\in\Z_+$ and $\Omega$ is the union
of $\Omega_n$. Although $d_0$ depends on the choice of
$\{\Omega_n\}$, the topology defined by this metric does not depend
on this choice. If $\Omega=[0,1]^k$ or $\Omega=\R^k$ and $\mu$ is
the Lebesgue measure, we omit the notation for the underlying
measure and $\sigma$-algebra and simply write $L_p([0,1]^k)$ or
$L_p(\R^k)$. We also replace $L_p([0,1])$ by $L_p[0,1]$. Note
\cite{kal} that $X=L_p[0,1]$ for $0\leq p<1$ is a separable infinite
dimensional $\F$-space satisfying $X'=\{0\}$. Moreover, for any
$p\in[0,1)$ and $k\in\N$, $L_p([0,1]^k)$ is isomorphic to $L_p[0,1]$
and $L_p(\R^k)$ is isomorphic to $L_p[0,1]$.

\begin{example}\label{ex1}Let $0<p<1$, $X=L_p([0,1]^k)$ and
$T_j\in L(X)$ be defined by the formula
$$
T_jf(x_1,\dots x_{j-1},x_j,x_{j+1}\dots,x_n)=f(x_1,\dots
x_{j-1},x_j/2,x_{j+1}\dots,x_n), \ \ \ 1\leq j\leq k.
$$
Then $\{e^{\langle t,T\rangle}\}_{t\in\K^k}$ is a uniformly
continuous and hereditarily hypercyclic operator group.
\end{example}

\begin{proof} The facts that $T_j$ are pairwise commuting,
$e^{\langle t,T\rangle}$ is well-defined for each $t\in\K^k$ and
$\{e^{\langle t,T\rangle}\}_{t\in\K^k}$ is a uniformly continuous
operator group are easily verified. Moreover, $T$ is an
EBS$_k$-tuple. Namely, $\KER T$ consists of all $f\in X$ vanishing
in a neighborhood of $(0,\dots,0)$ and therefore is dense. By
Corollary~\ref{kerim2}, $\{e^{\langle t,T\rangle}\}_{t\in\K^k}$ is
mixing. By Proposition~\ref{untr2}, $\{e^{\langle
t,T\rangle}\}_{t\in\K^k}$ is hereditarily hypercyclic.
\end{proof}

It is worth noting that the above example does not work for
$X=L_0([0,1]^k)$: $e^{\langle t,T\rangle}$ is not well-defined for
each non-zero $t\in\K^k$. Nevertheless, we can produce a strongly
continuous hereditarily hypercyclic operator group
$\{T_t\}_{t\in\R^k}$ on $L_0(\R^k)$.

\begin{example}\label{ex2}Let $k\in\N$, $X=L_0(\R^k)$ and for each
$t\in\R^k$, $T _t\in L(X)$ be defined by the formula
$T_tf(x)=f(x-t)$. Then $\{T_t\}_{t\in\R^k}$ is a strongly continuous
hereditarily hypercyclic operator group.
\end{example}

\begin{proof} The fact that $\{T_t\}_{t\in\R^k}$ is a strongly
continuous operator group is obvious. Pick a sequence
$\{t_n\}_{n\in\Z_+}$ of vectors in $\R^k$ such that $|t_n|\to
\infty$ as $n\to\infty$. Clearly the space $E$ of functions from $X$
with bounded support is dense in $X$. It is easy to see that
$T_{t_n}f\to 0$ and $T_{t_n}^{-1}f=T_{-t_n}f\to 0$ for each $f\in
E$. Hence $\{T_{t_n}:n\in\Z_+\}$ satisfies the universality
criterion from \cite{bp}. Thus $\{T_{t_n}:n\in\Z_+\}$ is universal
and therefore $\{T_t\}_{t\in\R^k}$ is hereditarily hypercyclic.
\end{proof}

Since $L_p([0,1]^k)$ and $L_p(\R^k)$ are isomorphic to $L_p[0,1]$,
we obtain the following corollary.

\begin{corollary}\label{ex11}Let $k\in\N$ and $0\leq p<1$. Then
there exists a hereditarily hypercyclic strongly continuous operator
group $\{T_t\}_{t\in\R^k}$ on $L_p[0,1]$.
\end{corollary}

Ansari \cite{ansa1} asked whether $L_p[0,1]$ for $0\leq p<1$ support
hypercyclic operators. This question was answered affirmatively by
Grosse--Erdmann \cite[Remark~4b]{ge1}. Corollary~\ref{ex11} provides
a 'very strong' affirmative answer to the same question. Finally, we
would like to mention a class of topological vector spaces very
different from the spaces in $\M$ in terms of operator semigroups.
Recall that operator semigroups from Theorem~\ref{saan} on spaces
$X\in\M$ depend analytically on the parameter: the map $t\mapsto
f(T_tx)$ from $\K^k$ to $\K$ is analytic for any $x\in X$ and $f\in
X'$.

\begin{proposition}\label{clU}Let a locally convex space $X$ be the
union of a sequence of its closed linear subspaces
$\{X_n\}_{n\in\Z_+}$ such that $X_n\neq X$ for each $n\in\Z_+$.
Assume also that $\{T_t\}_{t\in\R_+}$ is a strongly continuous
operator semigroup such that the function $t\mapsto f(T_tx)$ from
$\R_+$ to $\K$ is real-analytic for any $x\in X$ and $f\in X'$. Then
$\{T_t\}_{t\in\R_+}$ is non-cyclic.
\end{proposition}

\begin{proof}Let $x\in X$. Clearly $\R_+$ is the union of closed sets
$A_n=\{t\in\R_+:T_tx\in X_n\}$ for $n\in\Z_+$. By the Baire theorem,
there is $n\in\Z_+$ such that $A_n$ contains an interval $(a,b)$.
Now let $f\in X'$  be such that $X_n\subseteq\ker f$. Then the
function $t\mapsto f(T_tx)$ vanishes on $(a,b)$. Since this function
is analytic, it is identically $0$. That is, $f(T_tx)=0$ for any
$t\in\R_+$ and any $f\in X'$ vanishing on $X_n$. By the Hahn--Banach
theorem, $T_tx\in X_n$ for each $t\in\R_+$. Hence $x$ is not cyclic
for $\{T_t\}_{t\in\R_+}$. Since $x$ is arbitrary,
$\{T_t\}_{t\in\R_+}$ is non-cyclic.
\end{proof}

Note that a countable locally convex direct sum of infinite
dimensional Banach spaces may admit a hypercyclic operator
\cite{fre}. This observation together with the above proposition
make the following question more intriguing.

\begin{question}\label{sum} Let $X$ be the locally convex
direct sum of a sequence of separable infinite dimensional Banach
spaces. Does $X$ admit a hypercyclic strongly continuous semigroup
$\{T_t\}_{t\in\R_+}$?
\end{question}

\subsection{A question by Berm\'udez, Bonilla, Conejero and Peris}

Using \cite[Theorem~2.2]{bama-book} and Theorem~\ref{mi00}, one can
easily see that if $T$ is an extended backward shift on a separable
infinite dimensional Banach space $X$, then both $I+T$ and $e^T$ are
hereditarily hypercyclic. Clearly, an extended backward shift $T$
has dense range and dense generalized kernel $\Ker
T=\bigcup\limits_{n=1}^\infty \ker T^n$. The converse is not true in
general. This leads to the following question.

\begin{question}\label{ququ}
Let $T$ be a continuous linear operator on a separable Banach space,
which has dense range and dense generalized kernel. Is it true that
$I+T$ and/or $e^T$ are mixing or at least hypercyclic?
\end{question}

This reminds of the following question \cite{holo} by Berm\'udez,
Bonilla, Conejero and Peris.

\medskip
\noindent {\bf Question~B$^{\bf 2}$CP.} \ \it Let $X$ be a complex
Banach space and $T\in L(X)$ be such that its spectrum $\sigma(T)$
is connected and contains $0$. Does hypercyclicity of $I+T$ imply
hypercyclicity of $e^T$? Does hypercyclicity of $e^T$ imply
hypercyclicity of $I+T$? \rm
\medskip

We shall show that the answer to both parts of the above question is
negative. Before doing this, we would like to raise a similar
question, which remains open.

\begin{question} \label{qqq2} Let $X$ be a Banach
space and $T\in L(X)$ be quasinilpotent. Is hypercyclicity of $I+T$
equivalent to hypercyclicity of $e^T$?
\end{question}

If the answer is affirmative, then the following interesting
question naturally arises.

\begin{question}\label{q3} Let $T$ be a quasinilpotent bounded
linear operator on a complex Banach space $X$ and $f$ be an entire
function on one variable such that $f(0)=f'(0)=1$. Is it true that
hypercyclicity of $f(T)$ is equivalent to hypercyclicity of $I+T$?
\end{question}

We introduce some notation. Let $\T=\{z\in\C:|z|=1\}$,
$\D=\{z\in\C:|z|<1\}$, $\H^2(\D)$ be the Hardy Hilbert space on the
unit disk and $\H^\infty(\D)$ be the space of bounded holomorphic
functions $f:\D\to\C$. It is well-known that for $\alpha\in
\H^\infty(\D)$, the multiplication operator $M_\alpha
f(z)=\alpha(z)f(z)$ is a bounded linear operator on $\H^2(\D)$. It
is also clear that $\sigma(M_\alpha)=\overline{\alpha(\D)}$.

Godefroy and Shapiro \cite[Theorem~4.9]{gs} proved that if
$\alpha\in \H^\infty(\D)$ is not a constant function, then the
Hilbert space adjoint $M_\alpha^\star$ is hypercyclic if and only if
$\alpha(\D)\cap \T\neq\varnothing$. Moreover, they proved
hypercyclicity by means of applying the {\it Kitai Criterion}
\cite{kitai,kc}, which automatically \cite{gri} provides hereditary
hypercyclicity. Thus their result can be stated in the following
form.

\begin{proposition}\label{shi1} Let $\alpha\in \H^\infty(\D)$ be non-constant.
Then $M_\alpha^\star$ is hereditarily hypercyclic if\break
$\alpha(\D)\cap \T\neq\varnothing$ and $M_\alpha^\star$ is
non-hypercyclic if $\alpha(\D)\cap \T=\varnothing$.
\end{proposition}

We show that the answer to both parts of Question~B$^2$CP is
negative. Consider $U\subset \C$, being the interior of the triangle
with vertices $-1$, $i$ and $-i$. That is, $U=\{a+bi:a,b\in\R,\
a<0,\ b-a<1,\ b+a>-1\}$. Next, let $V=\{a+bi:a,b\in\R,\ 0<b<1,\
|a|<1-\sqrt{1-b^2}\}$. The boundary of $V$ consists of the interval
$[-1+i,1+i]$ and two circle arcs. Clearly, $U$ and $V$ are bounded,
open, connected and simply connected. Moreover, $(1+U)\cap
\T\neq\varnothing$, where $1+U=\{1+z:z\in U\}$ and $e^U=\{e^z:z\in
U\}\subseteq \D$. Similarly, $(1+V)\cap \D=\varnothing$ and $e^V\cap
\T\neq\varnothing$. By the Riemann Theorem \cite{mark}, there exist
holomorphic homeomorphisms $\alpha:\D\to U$ and $\beta:\D\to V$.
Obviously $\alpha,\beta\in \H^\infty(\D)$ and are non-constant.
Since $I+M_\alpha^\star=M_{1+\alpha}^\star$,
$e^{M_\beta^\star}=M^\star_{e^\beta}$ and both $(1+\alpha)(\D)=1+U$
and $e^\beta(\D)=e^V$ intersect $\T$, Proposition~\ref{shi1} implies
that $I+M_\alpha^\star$ and $e^{M_\beta^\star}$ are hereditarily
hypercyclic. Since $I+M_\beta^\star=M_{1+\beta}^\star$,
$e^{M_\alpha^\star}=M^\star_{e^\alpha}$, $e^\alpha(\D)=e^U$ is
contained in $\D$ and $(1+\beta)(\D)=1+V$ does not meet
$\overline{\D}$, Proposition~\ref{shi1} implies that
$e^{M_\alpha^\star}$ and $I+M_\beta^\star$ are non-hypercyclic.
Finally, $\sigma(M_\alpha^\star)=\overline{U}$ and
$\sigma(M_\beta^\star)=-\overline{V}$. Hence the spectra of
$M_\alpha^\star$ and $M_\beta^\star$ are connected and contain $0$.
Thus we have arrived to the following result, which answers
negatively the Question~B$^2$CP.

\begin{proposition}\label{que} There exist bounded linear operators
$A$ and $B$ on a separable infinite dimensional complex Hilbert
space such that $\sigma(A)$ and $\sigma(B)$ are connected and
contain $0$, $I+A$ and $e^B$ are hereditarily hypercyclic, while
$e^A$ and $I+B$ are non-hypercyclic.
\end{proposition}

\bigskip

{\bf Acknowledgements.} \ The author is grateful to the referee for
useful comments and suggestions, which helped to improve the
presentation.

\small\rm

\vskip1truecm

\scshape

\noindent Stanislav Shkarin

\noindent Queens's University Belfast

\noindent Department of Pure Mathematics

\noindent University road, Belfast, BT7 1NN, UK

\noindent E-mail address: \qquad {\tt s.shkarin@qub.ac.uk}

\end{document}